\documentclass[final]{amsart}

\usepackage{Package}

\begin{document}

\title{Iterated Random Functions and Slowly Varying Tails}
\author{Piotr Dyszewski}
\thanks{The author was partially supported by the NCN grant DEC-2012/05/B/ST1/00692}

\address{Instytut Matematyczny, Uniwersytet Wrocławski, Plac Grunwaldzki 2/4, 50-384 Wrocław, Poland}
\email{piotr.dyszewski@math.uni.wroc.pl}
\urladdr{www.math.uni.wroc.pl/$\sim$pdysz}

\date{\today}


\maketitle

\begin{abstract}
   Consider a sequence of i.i.d. random Lipschitz functions $\{\Psi_n\}_{n \geq 0}$. Using this sequence we can define a~Markov chain via the recursive formula
   $R_{n+1} = \Psi_{n+1}(R_n)$. It is a~well known fact that under some mild moment assumptions this Markov chain has a unique stationary distribution.
   We~are interested in the tail behaviour of this distribution in the case when $\Psi_0(t) \approx A_0t+B_0$. We~will show that under subexponential 
   assumptions on the random variable $\log^+(A_0\vee B_0)$ the tail asymptotic in question can be described using the integrated tail function of $\log^+(A_0\vee B_0)$.
   In~particular we will obtain new results for the random difference equation $R_{n+1} = A_{n+1}R_n+B_{n+1}$.
\end{abstract}

\section{Introduction}
   Consider a sequence of independent identically distributed (i.i.d.) random Lipschitz functions $\{\Psi_n\}_{n \geq 0}$, where $\Psi_n \colon \RR \to \RR$ for $n \in \NN$. 
   Using this sequence we can define a~Markov chain via the recursive formula
   \begin{equation}\label{SDS}
      R_{n+1} = \Psi_{n+1}(R_n) \qquad \mbox{for }n \geq 0,
   \end{equation}
   where $R_0 \in \RR$ is arbitrary but independent of the sequence $\{\Psi_n\}_{n \geq 0}$. Put $\Psi=\Psi_0$. 
   We are interested in the existence and properties of the stationary distribution of the Markov chain 
   $\{R_n\}_{n\geq 0}$, that is the solution of the stochastic fixed point equation
   \begin{equation}\label{FixedPoint}
      R \stackrel{d}{=} \Psi(R) \qquad \mbox{$R$ independent of $\Psi$},
   \end{equation}
   where the distribution of random variable $R$ is the stationary distribution of the Markov chain $\{R_n\}_{n \geq 0}$. \bigskip

   The main example, we have in mind, is the random difference equation, where  $\Psi$ is an affine transformation, that is $\Psi_n(t) = A_nt+B_n$ with
   $\{(A_n,B_n)\}_{n\geq 0}$ being an i.i.d. sequence of two-dimensional random vectors. Then the formula \eqref{SDS} can be written as 
   \begin{equation}\label{AnXn+Bn}
      R_{n+1} = A_{n+1}R_n+B_{n+1} \qquad \mbox{for }n \geq 0.
   \end{equation}
   Put $(A,B)=(A_0,B_0)$.
   It is a well known fact that if 
   \begin{equation*}
      \EE[\log|A|]< 0\quad \mbox{and}\quad \EE[\log^+|B|]<\infty,
   \end{equation*} 
   then the Markov chain $\{R_n\}_{n\geq 0}$ given by \eqref{AnXn+Bn} has a~unique stationary 
   distribution which can be represented as 
   the distribution of the random variable
   \begin{equation}\label{PerpetuityDef}
      R = \sum_{n \geq 0}{B_{n+1}\prod_{k=1}^n{A_k} },
   \end{equation}
   for details see \cite{vervaat1979stochastic}. 
   Random variables of this form can be found in analysis of probabilistic algorithms or financial mathematics, where $R$ would be called a~perpetuity. 
   Such random variables occur also in number theory, combinatorics, as a solution to stochastic fixed point equation
   \begin{equation}\label{RDE}
      R \stackrel{d}{=} AR+B \qquad R\mbox{ independent of }(A,B),
   \end{equation}
   atomic cascades, random environment branching processes, exponential functionals of L\'evy processes, Additive Increase Multiplicative Decrease algorithms \cite{guillemin2004}, 
   COGARCH processes \cite{ContinuousTime}, and more.
   A~variety of examples for possible applications of $R$ can be found in  \cite{GoldieGrubel1996, Grey1994, PerpetuitiesandRandomEquations}.\bigskip 

   From the application point of view, the key information is the behaviour of the tail of $R$, that is 
   \begin{equation*}
      \PP[R > x] \qquad \mbox{as } x \to \infty.
   \end{equation*}
   This problem was investigated by various authors, for example by Goldie and Grübel~\cite{GoldieGrubel1996} and in a similar setting by Hitczenko 
   and~Wesołowski~\cite{HitczenkoWesolowski2009}. The first result says that if $B$ is bounded, $\PP[A\in [0,1]]=1$ and the distribution of $A$ 
   behaves like the uniform distribution in the 
   neighborhood of $1$,  then $R$ given by \eqref{PerpetuityDef} has thin tail, more precisely $\log \PP[R \geq x] \sim -c x \log(x)$.
   Recall that for two positive 
   functions $f(\cdot)$ and $g(\cdot)$, by $f(x) \sim g(x)$ we mean that $\lim_{x \to \infty}{f(x)/g(x)} = 1$. In this paper we are only interested in limits as $x \to \infty$,
   so from now we omit the specification of the limit.

   There is also the result of Kesten~\cite{Kesten1973} and later on, in the same setting, of Goldie~\cite{Goldie1991}. 
   The~essence of this result is that under Cram\'er's condition, that is if
   $\EE[|A|^{\alpha}]=1$ for some $\alpha>0$ such that $\EE[|B|^{\alpha}] < \infty$, the tail of $R$ is regularly varying, i. e. $\PP[R>x] \sim cx^{-\alpha}$ 
   for some positive and finite 
   constant $c$ and $R$ defined by~\eqref{PerpetuityDef}.

   Finally, the result of Grincevičius~\cite{Grin1975}, which was later generalised by Grey~\cite{Grey1994}, states that in the case of positive $A$ 
   if for some $\alpha>0$ we have
   $\EE[A^{\alpha}]<1$ and $\PP[B>x] \sim x^{-\alpha}L(x)$, where $L$ is slowly varying (that is $L(cx) \sim L(x)$ for any positive $c$), then the tail of $R$ is 
   again regularly 
   varying, in fact $\PP[R>x] \sim cx^{-\alpha}L(x)$. Note that in this case the tail of perpetuity $R$ exhibits the same rate of decay as the tail of the input, that is 
   $\PP[R >x] \sim c\PP[B>x]$. \bigskip

   However, in the case when $\PP[A>x]$ or $\PP[B>x]$ is a slowly varying function of $x$, up to our knowledge, little is known about the behaviour of $\PP[R>x]$  as 
   $x \to \infty$. 
   This is the problem we consider in the present paper.\bigskip 

   The case of general fixed point equation \eqref{FixedPoint} was studied by Goldie~\cite{Goldie1991}, where several particular forms 
   of the transformation $\Psi$ were treated.
   Later Mirek~\cite{Mirek2011} found the tail asymptotic of the solution of \eqref{FixedPoint} with $\Psi$ being Lipschitz such that $\Psi(t) \approx \mbox{Lip}(\Psi)t$,
   where ${\rm Lip}(\Psi)$ is the Lipschitz constant. The result says that if $\EE[\log( {\rm Lip}(\Psi) )] < 0$ and 
   $\EE[{\rm Lip}(\Psi)^{\alpha}]=1$ for some $\alpha>0$, 
   then $R$ solving \eqref{FixedPoint}
   exhibits regularly varying tail $\PP[|R|>x] \sim c x^{\alpha}$. Grey~\cite{Grey1994} also treated generalized fixed point equations \eqref{FixedPoint} in 
   the setting introduced by Grincevičius~\cite{Grin1975}.

   It turns out that the assumption $\EE[\log({\rm Lip}(\Psi))]<0 $ is necessary for the existence of the probabilistic 
   solutions of \eqref{FixedPoint}. For the existence and asymptotic behaviour of the invariant measure of the Markov chain \eqref{SDS} in the critical case, 
   that is $\EE[\log({\rm Lip}(\Psi))]=0$, see \cite{BBE1997, Buraczewski2007, BBD2012, BB2014}.\bigskip

   This paper gives an answer to the question about asymptotic of $\PP[R>x]$, where $R$ solves \eqref{FixedPoint}, in the case of slowly varying input.
   Assuming that the Lipschitz function $\Psi$ satisfies
   \begin{equation*}
     At+B-D \leq \Psi(t)\leq At^++B^++D \qquad \mbox{for } t \in \RR,
   \end{equation*}
   with $D>0$ being relatively small and $A>0$, 
   we will show that under subexponential assumptions on the random variable $\log(A \vee B)$ one has 
   \begin{equation*}
      \PP[R > x] \asymp \int_{\log(x)}^{\infty}{\PP[\log(A\vee B) >y] \: \ud y}.
   \end{equation*}
   Recall that for two positive functions $f(\cdot)$, $g(\cdot)$ by $f(x) \asymp g(x)$ we mean that $g(x)=O(f(x))$ and $f(x) = O(g(x))$. 
   Furthermore, in our setting, the integral expression on the right hand side will be a~slowly varying function of $x$.
   Moreover in several cases we will establish a~precise tail asymptotic of $R$.
   In~order to obtain full description of tail 
   behaviour for the sequence $\{R_n\}_{n \geq 0}$ we will study finite time horizon. We will show that if distribution of $\log(A\vee B)$ is subexponential, then it holds true that
   \begin{equation*}
      \PP[R_n>x] \asymp n \PP[A\vee B>x],
   \end{equation*}
   where $\{R_n\}_{n\geq 0}$ in given by \eqref{SDS}.
   \bigskip

   The main result gives description of tail asymptotic of the solution of~\eqref{RDE} and also
   \begin{equation*}
      R \stackrel{d}{=} AR^++B \qquad R\mbox{ independent of }(A,B),
   \end{equation*}
   which is closely related to the ruin probability, 
   for details see \cite{Collamore2009}. We can also obtain a~description of the solutions to 
   \begin{equation*}
      R \stackrel{d}{=} A_1 |R| + \sqrt{D+A_2R^2}\qquad R \mbox{ independent of } (A_1,A_2,D),
   \end{equation*}
   where $\PP[D>x] = o\left(\PP\left[A_1+\sqrt{A_2} >x\right]\right)$.
   This corresponds to an autoregressive process with ARCH$(1)$ errors, which was described by Borkovec and Klüppelberg \cite{borkovec2001}.
   To find the behaviour of $\PP[|R|>x]$ just take $\Psi(t) = | A_1 t + \sqrt{D+A_2 (t^+)^2}|$. \bigskip 


   The paper is organised as follows: In the second section we will briefly recall basic definitions and properties of subexponential distributions, after that in the third section we will 
   present a~precise 
   statement of the result followed by some remarks and sketch of the proof. Finally, in the last fourth section, we will give the full proof of the results.

\section{Subexponential Distributions}

In this section we will recall well known notions from the theory of heavy-tailed distributions. 
Next we will quote a theorem about tail behaviour of a maxima of perturbed random walk, which will be particularly useful in the proof of the main result. 
Firstly, for a~distribution $F$ on $\RR$ we define tail function $\Fbar$ by the formula $\Fbar(x) = 1-F(x)$ for $x \in \RR$.

\begin{df}\label{DefL}
   A distribution $F$ on $\RR$ is called long-tailed if $\Fbar(x)>0$ for all $x \in \RR$ and for any fixed $y \in \RR$  
   \begin{equation*}\label{PropL}
      \overline{F}(x+y) \sim \overline{F}(x).
   \end{equation*}
   We denote the class of long-tailed distributions by $\Lcal$.
\end{df}

Notice that if $F \in \Lcal$ then the function $x \mapsto \Fbar\left(\log(x)\right)$ is slowly varying as $x \to \infty$. Therefore one can use Potter's Theorem 
(see \cite{bingham1989regular}: Theorem 1.5.6) to obtain the following corollary.

\begin{cor}\label{Potter}
   If $F \in \Lcal$, then for any chosen $\Delta>1$ and $\delta >0$ there exists $X=X(\Delta,\delta)$ such that
   \begin{equation*}
      \frac{\Fbar(x)}{\Fbar(y)} \leq \Delta e^{\delta |x-y|} \qquad \mbox{for } x,y \geq X.
   \end{equation*}
\end{cor}

It turns out that class $\Lcal$ is too big for our purposes. More precisely, we will need distributions satisfying some convolution properties. Recall that  $F^{*2}$ stands for the~twofold 
convolution of the distribution $F$. 

\begin{df}\label{DefSR}
   A distribution $F$ on $\RR$ is called subexponential if $F \in \mathcal{L}$ and
   \begin{equation*}
      \overline{F^{*2}}(x) \sim 2\overline{F}(x).
   \end{equation*}   
   The class of subexponential distributions will be denoted by $\Scal$.
\end{df}

Note that if $X_1$ and $X_2$ are i.i.d. with distribution $F \in \Scal$, then by the definition above 
\begin{equation*}
   \PP[X_1+X_2>x] \sim 2\PP[X_1>x] \sim \PP[X_1\vee X_2>x].
\end{equation*}
This is a type of phenomena that we want to use in the near future.
We see that $\Scal \subset \Lcal$ and it is a~well known fact that this inclusion is proper. 
For examples of distributions in $\Lcal \setminus \Scal$ see \cite{JAZ:4902368} or \cite{JAZ:4902980}.
The~following proposition is a well known fact which will be useful thought the proofs of the results. We follow the statement presented in~\cite{IntSubExp}. 

\begin{prop}\label{PropFoss}
   Suppose that $F \in  \Scal$ . Let $G_1 , \ldots , G_n$ be distributions such that $\overline{G}_i (x) \sim c_i \overline{F}(x)$  for some constants $c_i \geq 0$, $i = 1, \ldots , n$. 
   Then
   \begin{equation*}
      \overline{G_1 ∗ \ldots ∗ G_n} (x) \sim (c_1 + \ldots + c_n )\overline{F} (x).
   \end{equation*}
   If $c_1 + \ldots + c_n > 0$, then $G_1 ∗\ldots ∗ G_n \in \Scal$.
\end{prop}

The following theorem by Palmowski and Zwart~\cite{palmowski2007tail} is crucial for our future purposes. The result itself deals with i.i.d. sequence  
 $\{(X_n, Y_n)\}_{n \geq 0}$, where $X_n$ 
are i.i.d. increments of negatively driven stochastic process and $Y_n$ being maxima of this process taken at some renewal epochs. Nevertheless the result 
and the proof presented in~\cite{palmowski2007tail} remains valid for arbitrary i.i.d. sequence $\{(X_n, Y_n) \}_{n\geq 0}$.

\begin{thm} \label{ZPBZ}
   Let $\{ (X_n,Y_n) \}_{n \geq 0}$ be a sequence of i.i.d. two-dimensional random vectors such that $\EE[X_1] < 0$ and $\EE[X_1\vee Y_1] < \infty$.
   Assume that distribution on $\RR_+$ given by the tail function
   \begin{equation*}
      x \mapsto {1\wedge \int^{\infty}_x{\PP[X_1\vee Y_1> y]\: \ud y}}
   \end{equation*}
   is subexponential. Then
   \begin{equation*}
      \PP\left[ \sup_{n \geq 0}\left\{ Y_{n+1} + \sum_{j=1}^n{X_j}\right\} >x   \right] \sim -\frac{1}{\EE[X_1]} \int^{\infty}_x{\PP[X_1\vee Y_1> y]\: \ud y}.
   \end{equation*}
\end{thm}

The $\RR_+$ in the theorem above and for the rest of the paper stands for $[0,+\infty)$. For conditions on $F$ guaranteeing subexponentiality of distribution given by the tail function 
$x~\mapsto~{1 \wedge \int^{\infty}_x{\Fbar(y)\: \ud y}}$ see~\cite{Kluppelberg1988}.

\section{Main Result}
   In this section we will give a precise statement of the main result of the paper followed by some remarks and idea behind the proof. 

\subsection{Statement}
   Recall that we consider a Markov chain $\{R_n\}_{n\geq 0}$ given by \eqref{SDS}, where for each $n \in \NN$ the function $\Psi_n \colon \RR\to\RR$ satisfies 
   \begin{equation}\label{ConditionsOptimal}
      A_nt+B_n-D_n \leq \Psi_n(t) \leq A_nt^+ +B^+_n+D_n \qquad \mbox{for } t \in \RR
   \end{equation}
   and some random variables $A_n$, $B_n$ and $D_n$. We are assuming that $\{ (\Psi_n,A_n,B_n,D_n) \}_{n \geq 0}$ are i.i.d., where $\Psi_n$ are Lipschitz functions 
   with
   \begin{equation}\label{DefLip}
      {\rm Lip}(\Psi_n) = \sup_{t_1\neq t_2} \left|\frac{\Psi_n(t_1) - \Psi_n(t_2)}{t_1-t_2} \right|. 
   \end{equation}
   Put $(\Psi,A,B,D)=(\Psi_0,A_0,B_0,D_0)$.
   From now our standing assumptions will be
   \begin{equation}\label{Standing} 
         A,D \geq 0 \mbox{ a.s.}, \qquad \EE[\log(A)] > -\infty, \qquad \EE[\log({\rm Lip}(\Psi))] < 0, \qquad  \EE[\log^+|B\pm D|] < \infty.
   \end{equation}
   Recall that $\log^+(x) = \log(x \vee 1)$. 
   Note that \eqref{ConditionsOptimal} implies  
   \begin{equation*}
      A_n \leq {\rm Lip}(\Psi_n)
   \end{equation*}
   and hence $\EE[\log(A)] <0$.
   For infinite time horizon, that is the case of the stationary distribution, we will also need to assume
   \begin{equation}\label{MomentsNew} 
         \EE[\log^+(A\vee B)^{1+\gamma}]< \infty \quad \mbox{for some }\gamma >0.
   \end{equation}
   In order to ensure that the stationary distribution has right-unbounded support we will need to assume the following tail behaviour 
   \begin{equation}\label{Tail}
      \PP[A\vee (B\pm D) > x] \sim \PP[A \vee B >x], \qquad \PP[A > x, \: B-D \leq -x] = o(\PP[A\vee B>x]).
   \end{equation}
   Define a probability distribution $F_I$ on $\RR_+$ via its tail function $\Fbar_I$ which is given by 
   \begin{equation}\label{ITF}
      \Fbar_I(x) = 1 \wedge \int_x^{\infty}{\PP[\log(A \vee B) >y]\: \ud y}.
   \end{equation}
   \begin{thm} \label{MainThm} 
      Assume that conditions \eqref{ConditionsOptimal},\eqref{Standing}, \eqref{MomentsNew} and \eqref{Tail} are satisfied and that $F_I$ defined by \eqref{ITF} is subexponential. 
      Then the Markov chain $\{R_n\}_{n\geq 0}$ given by \eqref{SDS} converges in distribution to a~unique stationary distribution which is a unique solution of \eqref{FixedPoint}. Furthermore
      \begin{equation}\label{WeakR}
      -\frac{\PP[R> 0]}{\EE[\log(A)]}\leq\liminf_{x\to\infty}\frac{\PP[R>x]}{\Fbar_I(\log(x))}\leq\limsup_{x \to\infty}{\frac{\PP[R>x]}{\Fbar_I(\log(x))}} \leq -\frac{1}{\EE[\log(A)]}.
      \end{equation}
      In particular, if $B-D > 0$ a.s., then 
      \begin{equation}\label{Bgeq0}
         \PP[R>x] \sim -\frac{1}{\EE[\log(A)]}\int_{\log(x)}^{\infty}{\PP[\log(A\vee B) > y] \: \ud y }. 
      \end{equation}
      Moreover, if
      \begin{itemize}
         \item $\PP[A>x] = o (\PP[B>x])$ then
               \begin{equation*}
                  \PP[R> x] \sim -\frac{1}{\EE[\log(A)]} \int_{\log(x)}^{\infty}{\PP[\log^+(B) > y] \: \ud y }, 
               \end{equation*}
         \item $\PP[B>x] = o (\PP[A>x])$ then
               \begin{equation*}
                  \PP[R> x] \sim -\frac{\PP[R > 0]}{\EE[\log(A)]} \int_{\log(x)}^{\infty}{\PP[\log(A) > y] \: \ud y }.
               \end{equation*}
      \end{itemize}
   \end{thm}

   Since in last two cases of the above theorem we obtain $\PP[R>x] \sim c\Fbar_I(\log(x))$ with $F_I \in \Scal\subseteq \Lcal$ and some constant $c$ we see that in each case 
   the distribution of $R$ exhibits slowly varying tail. 

   \begin{rem}\label{RemarkSlow}
      From the proof of the Theorem \ref{MainThm} one can see that in order to establish the lower bound 
      in \eqref{WeakR} one only uses the fact that the distribution of the random variable $A\vee B$ has a~slowly varying tail. 
      Precisely, assume \eqref{ConditionsOptimal}, \eqref{Standing}, \eqref{Tail} and that the function $x \mapsto \PP[A \vee B >x]$ is slowly varying. Then
      \begin{equation*}
         -\frac{\PP[R> 0]}{\EE[\log(A)]}\leq\liminf_{x\to\infty}\frac{\PP[R>x]}{\Fbar_I(\log(x))}
      \end{equation*}
      where the function $\Fbar_I$ is given by \eqref{ITF}. Since when $F_I \in \Lcal$ it is true that 
      $\PP[A \vee B >x] = \PP[\log(A \vee B) > \log(x)] = o(\Fbar_I(\log(x)))$ and we can also conclude that $\PP[A\vee B >x] = o(\PP[R>x])$.
   \end{rem}

   In order to obtain an extensive description of the asymptotic properties of the  Markov chain $\{R_n\}_{n\geq 0}$ given by \eqref{SDS} we will also investigate 
   the tail behaviour of random variables $R_n$ for finite~$n$.
   Put
   \begin{equation}\label{FDef}
      \Fbar(x)=\PP\left[\log(A\vee B) >x\right].
   \end{equation}
   It turns out that in case of finite time horizon one can obtain result analogous to Theorem~\ref{MainThm}. 

   \begin{thm}\label{PropFin}
      Assume \eqref{ConditionsOptimal}, \eqref{Standing}, \eqref{Tail}, $0 \leq n<\infty$ and that $F$ defined by \eqref{FDef} is subexponential. Assume additionally that
      \begin{equation}\label{Wcondition}
         \PP[R_0>x] \sim w \PP[A \vee B > x]
      \end{equation}
      for some constant $w \geq 0$. Then
      \begin{equation}\label{WeakFinU}
         w+\sum_{k=0}^{n-1}{\PP[R_k>0]} \leq\liminf_{x\to\infty}\frac{\PP[R_{n}>x]}{\PP[A \vee B > x]}\leq
             \limsup_{x \to\infty}{\frac{\PP[R_{n}>x]}{\PP[A \vee B >x ]}} \leq w+n.
      \end{equation}
      In particular if $R_0>0$ a.s. and $B-D>0$ a.s. then 
      \begin{equation*}
         \PP[R_n >x] \sim (w+n) \PP[A\vee B>x].
      \end{equation*}
      Furthermore if
      \begin{itemize}
         \item $\PP[A>x] = o (\PP[B>x])$ then
               \begin{equation}\label{FinBUW}
                        \PP[R_n>x]\sim  (w+n)\PP[B > x], 
               \end{equation}
         \item $\PP[B>x] = o (\PP[A>x])$ then
               \begin{equation}\label{FinAU}
                  \PP[R_n>x] \sim  \left( w+\sum_{k=0}^{n-1}{\PP[R_k>0]}\right)\PP[A>x].
               \end{equation}
      \end{itemize}       
   \end{thm}

   \begin{rem}
      Assume \eqref{Wcondition}, \eqref{ConditionsOptimal}, \eqref{Tail}, \eqref{Standing}, $0 \leq n<\infty$, and that the function 
      $ x \mapsto \PP[A\vee B >x]$ is slowly varying. Then
      \begin{equation*}
         w+\sum_{k=0}^{n-1}{\PP[R_k>0]} \leq \liminf_{x \to \infty }{ \frac{ \PP[R_n > x]  }{\PP[A \vee B >x]}  }. 
      \end{equation*}
   \end{rem}



\subsection{Random Difference Equation}\todo{Dopisac RDE}
   Suppose, for the rest of this section, that $\Psi(t)=At+B$ and $D=0$.
   In the case when $B > 0$ a.s., Theorem~\ref{MainThm} gives a~description of the tail of $R$ in terms of the distribution of $A \vee B$,
   which allows us to present an~example showing that in the case when $\PP[A>x]\sim \PP[B>x]$, the information about marginal distributions of $A$ and $B$ is not enough to determine the 
   tail asymptotic of $R$.
   \begin{ex}\label{OnlyEx}
      Fix a distribution $F$ on $\RR_+$ and consider two types of input: First one $\left(A^{(1)},B^{(1)}\right)$: with $A^{(1)} = B^{(1)}$ with distribution $F$. Then, assuming that
      the assumptions are satisfied, Theorem~\ref{MainThm} states in \eqref{Bgeq0}
      that the corresponding perpetuity $R^{(1)}$ satisfies 
      \begin{equation*}
         \PP\left[R^{(1)}>x\right] \sim -\frac{1}{\EE\left[\log\left(A^{(1)}\right)\right]} \int_{\log(x)}^{\infty}{\PP\left[\log\left(A^{(1)}\right) > y\right] \: \ud y }.
      \end{equation*}
      If now we consider the second type of input, namely $\left(A^{(2)}, B^{(2)}\right)$ where $A^{(2)}$, $B^{(2)}$ are independent with the same distribution $F$, 
      Theorem~\ref{MainThm} states that the corresponding
      perpetuity $R^{(2)}$ satisfies 
      \begin{equation*}
         \PP\left[R^{(2)}>x\right] \sim -\frac{1}{\EE\left[\log\left(A^{(2)}\right)\right]} \int_{\log(x)}^{\infty}{\PP\left[\log\left(A^{(2)}\vee B^{(2)}\right) > y\right] \: \ud y }
      \end{equation*}
      and since $A^{(1)} \stackrel{d}{=} A^{(2)}$ we can write
      \begin{equation*}
         \PP\left[A^{(2)}\vee B^{(2)} >x  \right] \sim 2 \PP\left[A^{(2)}>x \right] = 2\PP\left[A^{(1)}>x\right]
      \end{equation*}
      and we see that
      \begin{equation*}
         \PP\left[R^{(2)}>x\right] \sim 2\PP\left[R^{(1)}>x \right].
      \end{equation*}
      Even though the marginal distributions of the two types of input are exactly the same, the corresponding perpetuities have different tail asymptotic.  
   \end{ex}

   The main result of this paper is closely related to Theorem 4.1 by Maulik and Zwart~\cite{Maulik2006156} where so-called exponential functional of 
   L\'evy process is treated, i. e. a random variable of the form 
   $\int_0^{\infty} {e^{\xi_s} \: \ud s}$ where $\{\xi_s \:| \:  s \geq 0 \}$ is a L\'evy proses with negative drift. 
   Note that this is a perpetuity 
   corresponding to 
   \begin{equation*}
   A=e^{\xi_1} \quad\mbox{and}\quad B= \int_0^1{e^{\xi_s} \: \ud s}.
   \end{equation*}
   The theorem in question states that
   \begin{equation*}
      \PP\left[ \int_0^{\infty}{e^{\xi_s} \: \ud s}>x \right] \sim -\frac{1}{\EE[\xi_1]}\int_{\log(x)}^{\infty} { \PP[\xi_1>y] \: \ud y}
   \end{equation*}
   if $x \mapsto \int_x^{\infty}{ \PP[\xi_1>y] \: \ud y}$ is subexponential. 
   We see that Theorem 4.1 by Maulik and Zwart~\cite{Maulik2006156} is a particular case of the main result of this paper. Next example shows the importance of 
   second condition in \eqref{Tail}. 

   \begin{ex}
      Consider the input $(A,B)$ where $B=\ind_{[0,1]}(A)-A$. Assume that $A>0$ and $\EE[\log(A)]<0$. This ensures the existence of the solution $R$ to
      \begin{equation*}
         R \stackrel{d}{=} AR +\ind_{[0,1]}(A) -A \qquad R \mbox{ independent of } A.
      \end{equation*}
      We see that $\PP[B>0] = \PP[A \in [0,1)]>0$, but the solution is bounded. Indeed, notice that $R$ also satisfies 
      \begin{equation*}
         R-1 \stackrel{d}{=}A(R-1) + \ind_{[0,1]}(A) -1 \qquad R \mbox{ independent of } A
      \end{equation*}
      and so $R-1$ is a perpetuity obtained from the input $(A, \ind_{[0,1]}(A) -1)$. Since $\ind_{[0,1]}(A)-1 \leq 0$ a.s., we know that $R-1\leq 0$ a.s.
      Whence we can conclude that the perpetuity $R$ obtained from the input $(A,B)$ is bounded above by $1$ a.s. This is due to the fact that in this case 
      \begin{equation*}
         \PP[A>x, \: B \leq -x]  = \PP[A>x] = \PP[A\vee B >x] \qquad \mbox{for } x>1.
      \end{equation*}
   \end{ex}

   Theorem \ref{MainThm} is also related to results from \cite{Grin1975,Grey1994,rivero2012,InterplayJinzhu} where arising perpetuities exhibit the tail behaviour 
   similar to the tail behaviour of the input. The first one, for example, says that 
   \begin{equation*}
      \frac{\PP[R>x]}{\PP[B>x]} \sim  \frac{1}{1-\EE[A^{\alpha}]}
   \end{equation*}
   if $\EE[A^{\alpha}]<1$ and $\PP[B>x]\sim x^{-\alpha}L(x)$ for some slowly varying function $L$ and $\alpha >0$.
   We see that when $\alpha \to 0$ the constant $(1-\EE[A^{\alpha}])^{-1}$ tends to infinity. 
   Theorem \ref{MainThm} corresponds to the case with $\alpha =0$ and tells us what is the proper 
   asymptotic. This also gives the reason for the blowup of the constant. 
   By Remark~\ref{RemarkSlow}, $\PP[A\vee B >x] = o(\PP[R>x])$ and we can write 
   \begin{equation*}
      \frac{\PP[R>x]}{\PP[B>x]} \geq \frac{\PP[R>x]}{\PP[A\vee B>x]} \to \infty \qquad \mbox{as } x \to \infty.
   \end{equation*}
   \bigskip

\subsection{Idea of the proof}
   The key problem is to understand the random difference equation, i. e. the case 
   $\Psi(t) = At+B$. 
   For simplicity, we will focus on that case in the following discussion.  
   The~convolution property in Definition \ref{DefSR} of the subexponential distributions says that for 
   $X_1$ and $X_2$ independent with the same distribution $F \in \Scal$ it is true that $\PP[X_1+X_2 > x] \sim \PP[X_1\vee X_2 >x]$. It turn's out that the series \eqref{PerpetuityDef} exhibits 
   a~similar phenomena, more precisely we are able to approximate 
   \begin{equation*}
      \PP\left[\sum_{n \geq 0}{B_{n+1} \prod_{j=1}^n{A_j} } >x \right] \quad \mbox{by using} \quad \PP\left[\sup_{n \geq 0}{\left\{B_{n+1} \prod_{j=1}^n{A_j} \right\} } >x \right]. 
   \end{equation*}
   In order to achieve that we apply technique used in \cite{buraczewskiprecise, buraczewski2013}. 
   This technique revolves around the idea of grouping the terms of the series of the same order and investigating the sizes of the groups. 
   Then, after obtaining the above relation, we can interpret random variable $\sup_{n \geq 0}{B_{n+1} \prod_{j=1}^n{A_j} } $ as a~supremum  of a~perturbed 
   random walk and use the 
   known theory, namely Theorem \ref{ZPBZ},  to derive upper bound for the desired tail asymptotic. Next, adapting some classical techniques, used for example in \cite{palmowski2007tail}, 
   we get lower 
   bound for tail asymptotic. Roughly speaking, we find relatively big subsets of $\{R>x\}$ on which we have control over the whole sequence 
   $\left\{B_{n+1} \prod_{j=1}^n{A_j} \right\}_{n \geq 0} $.

\section{Proof}
   In this section we will prove the main result of the paper. Recall that we consider an i.i.d. sequence $\{ (\Psi_n, A_n, B_n, D_n) \}_{n \geq 0}$ such that
   $A_n>0$, $D_n\geq 0$ and
   \begin{equation*}
      A_nt+B_n-D_n \leq \Psi_n(t) \leq A_nt^++B_n^+ +D_n \qquad \mbox{for } n \geq 0 \mbox{ and } t \in \RR.
   \end{equation*}
   Put $(\Psi,A,B,D)=(\Psi_0,A_0,B_0,D_0)$ and let 
   \begin{equation*}
      \mu = - \EE[\log(A)].
   \end{equation*}
   Random walk generated by $\log(A)$ will be very useful, hence define
   \begin{equation}\label{SnDef}
      S_n=\sum_{j=1}^n{\log(A_j)} \qquad \mbox{for } n \geq 0
   \end{equation}
   and
   \begin{equation}\label{BnU}
      \overline{B}_n = (B_n^+ +D_n)\vee 1, \qquad \underline{B}_n = B_n-D_n \qquad \mbox{for } n \geq 0
   \end{equation}
   finally let $\overline{B}=\overline{B}_0$, $\underline{B}=\underline{B}_0$. Notice that \eqref{Tail} implies
   \begin{equation*}
      \PP[A\vee \underline{B}>x] \sim \PP[A \vee B > x] \sim \PP[A \vee \overline{B}>x].
   \end{equation*}
   For $k<n$ define the backward iterations of $\Psi$ by
   \begin{equation*}
      \Psi_{k:n}(t) = \Psi_k \circ \Psi_{k+1} \circ \ldots \circ \Psi_n(t).
   \end{equation*}
   We will use the convention that for $k >n$  $ \Psi_{k:n}(t) = t$.
   For $n \in \NN$ we can put
   \begin{equation*}
      \underline{\Psi}_n(t) = A_nt+\underline{B}_n \qquad \mbox{and} \qquad \overline{\Psi}_n(t)= A_nt^+ + \overline{B}_n
   \end{equation*} 
   and define $\underline{\Psi}_{k:n}$ and $\overline{\Psi}_{k:n}$ in the same manner as $\Psi_{k:n}$. 
   Notice that using this notation and the bounds on $\Psi_n(t)$, we get 
   \begin{equation*}
      \underline{\Psi}_n(t) \leq \Psi_n(t) \leq \overline{\Psi}_n(t)
   \end{equation*}
   and since $\overline{\Psi}$ and $\underline{\Psi}$ are monotone by iteration it gives 
   \begin{equation*}
      \underline{\Psi}_{k:n}(t) \leq \Psi_{k:n}(t) \leq \overline{\Psi}_{k:n}(t).
   \end{equation*}
   In particular
   \begin{equation*}
      \underline{\Psi}_{1:n}(t)= \sum_{k=0}^{n-1}{\underline{B}_{k+1} \prod_{j=1}^k{A_j}} + t \prod_{j=1}^{n}{A_j} \leq \Psi_{1:n}(t)
   \end{equation*}
   and
   \begin{equation*}
      \Psi_{1:n}(t)\leq \sum_{k=0}^{n-1}{\overline{B}_{k+1} \prod_{j=1}^k{A_j}} + t^+ \prod_{j=1}^{n}{A_j} = \overline{\Psi}_{1:n}(t).
   \end{equation*}
   We will use the following lemma quite often. The proof follows the idea presented in \cite{palmowski2007tail}.
\begin{lem}\label{ConeLemma}
      Assume \eqref{Standing} and for $\delta$, $K >0$ consider the sets
      \begin{equation}\label{EnDef}
         E_n=E_n(K, \delta)  = \{ S_j \in (-j( \mu +\delta) -K, -j(\mu -\delta) +K), \: j \leq n  \}
      \end{equation}
      and
      \begin{equation}\label{FnDef}
         F_n= F_n(K, \delta) = \left\{ \left|\underline{B}_j\right| \leq e^{\delta j +K}, \: j \leq n \right\}.
      \end{equation}
      Then the following claim holds
      \begin{equation}\label{ConeProp}
         \forall \delta, \varepsilon >0 \quad \exists K> 0 \quad \PP\left[ \bigcap_{j \geq 1}{(E_j\cap F_j)} \right] \geq 1-\varepsilon.
      \end{equation}
\end{lem}
\begin{proof}
   For $K$ large enough it is true that  $\PP\left[\log\left| \underline{B} \right|>K  \right]< 1/2$ and since for $y \in (0,1/2)$ it holds 
   that $\log(1-y) \geq -2y$, we can write 
   \begin{eqnarray*}
      \log(\PP[F_n]) & =    & \sum_{j=1}^{n}{\log(1-\PP\left[\log\left|\underline{B}_j\right| > \delta j +K \right])} \geq 
                                -2 \sum_{j=1}^n{\PP\left[\log\left|\underline{B}\right|>\delta j +K \right]}   \\ 
                     & \geq & -2 \sum_{j = 1}^{\infty}{ \PP[ \delta^{-1}( \log\left|\underline{B}\right|-K) > j]} \geq -2 \delta^{-1} \EE[(\log\left|\underline{B}\right|-K)_+]
   \end{eqnarray*}
   and so $\PP[F_n] \to 1$ as $K \to \infty$ uniformly with respect to $n$ since 
   \begin{equation*}
      \EE\left[\left(\log\left|\underline{B}\right|-K\right)_+\right]= \EE\left[\left(\log^+\left|\underline{B}\right|-K\right)_+\right]< \infty.
   \end{equation*}
   Combining this fact with the strong law of large numbers
   for the sequence $\{ S_n\}_{n \geq 0}$ we observe that we
   have shown that for any
   $\varepsilon, \delta > 0$ we can always take $K>0$ large enough such that
   \begin{equation*}
      \forall n \quad \PP[E_n \cap F_n] \geq 1-\varepsilon 
   \end{equation*}
   and since the sequence of sets $\{E_n\cap F_n\}_{n \geq 0}$ is decreasing in the sense of inclusion, we can conclude that
   \begin{equation*}
      \PP\left[ \bigcap_{j \geq 1}{(E_j\cap F_j)} \right] \geq 1-\varepsilon
   \end{equation*}
   and hence the proof is complete.
\end{proof}

   Note that the statement of the Lemma~\ref{ConeLemma} remains true if we replace $\underline{B}_j$ by $\overline{B}_j$ in the definition of the set $F_n$.
   The bounds on $\Psi$ imply that we can bound the solution of \eqref{FixedPoint} by two perpetuities, namely
   \begin{equation}\label{R+Def}
      \overline{R}=\sum_{n \geq 0}{\overline{B}_{n+1} \prod_{k=1}^n{A_j}}
   \end{equation}
   and 
   \begin{equation*}
      \underline{R}=\sum_{n \geq 0}{\underline{B}_{n+1} \prod_{k=1}^n{A_j}}.
   \end{equation*}
   The main idea of the proof is to approximate $\PP\left[\overline{R} >e^x \right]$ by using $\PP[M>x]$, where
   \begin{equation*}
      M = \sup_{n\geq 0}\left\{\log \left(\overline{B}_{n+1}\right) + \sum_{j=1}^{n}{\log(A_j)} \right\}.
   \end{equation*}
   Since $\overline{B}_1\geq 1$ we know that $ M > 0$ a.s. 
   Furthermore, we have $e^M \leq  \overline{R}$ and the last series is convergent a.s  by \eqref{Standing}.
   Having introduced this notation, we are ready to prove the main theorem.

\begin{proof}[Proof of the Theorem \ref{MainThm}]
   Fix large $x \in \RR$. The proof consists of five steps. \bigskip

\textit{Step 1: Existence, uniqueness and representation of the stationary distribution.} Note that
   \begin{equation*}
      R_n \stackrel{d}{=} \Psi_{1:n}(R_0)
   \end{equation*}
   so in order to prove that $\{ R_n\}_{n \geq 0}$ converges in distribution, it is sufficient to show that the sequence $\{ \Psi_{1:n}(R_0)  \}_{n \geq 0}$ 
   converges a.s.
   Recall that \eqref{ConditionsOptimal} implies  
   \begin{equation*}
      A_n \leq {\rm Lip}(\Psi_n) 
   \end{equation*}
   also, by the definition \eqref{DefLip} 
   \begin{equation*}
      {\rm Lip}(\Psi_{1:m}) \leq \prod_{j=1}^m{{\rm Lip}(\Psi_j)} \qquad \mbox{for } m \in \NN.
   \end{equation*}
   For $n\geq m$ and $t_1,\: t_2 \in \RR$ we have
   \begin{eqnarray*}
      | \Psi_{1:n}(t_1) - \Psi_{1:m}(t_2)| & \leq & {\rm Lip}(\Psi_{1:m}) | \Psi_{m+1:n}(t_1) - t_2| \leq {\rm Lip}(\Psi_{1:m})\left( |\Psi_{m+1:n}(t_1)| + |t_2|\right)\\
          & \hspace{-2cm} \leq &\hspace{-1cm}  {\rm Lip}(\Psi_{1:m})\left( \overline{\Psi}_{m+1:n}(t_1)\vee |\underline{\Psi}_{m+1:n}(t_1)| + |t_2|\right)\\
          & \hspace{- 2 cm}\leq &\hspace{ -1 cm} {\rm Lip}(\Psi_{1:m}) 
                        \left( \sum_{k=m}^{n-1}{(\overline{B}_{k+1}\vee |\underline{B}_{k+1}|)\prod_{j=m+1}^k{A_j} } + |t_1|\prod_{j=m+1}^n{A_j} +|t_2| \right)\\
          & \hspace{-2 cm} \leq & \hspace{ -1 cm} {\rm Lip}(\Psi_{1:m}) 
                        \left( \sum_{k=m}^{n-1}{(\overline{B}_{k+1}\vee |\underline{B}_{k+1}|)\prod_{j=m+1}^k{{\rm Lip}(\Psi_j)} } + |t_1|\prod_{j=m+1}^n{{\rm Lip}(\Psi_j)} +|t_2| \right)\\
          & \hspace{-2cm}\leq &\hspace{-1cm} \sum_{k=m}^{n-1}{(\overline{B}_{k+1}\vee |\underline{B}_{k+1}|)\prod_{j=1}^k{{\rm Lip}(\Psi_j)} } + 
                        |t_1|\prod_{j=1}^n{{\rm Lip}(\Psi_j)} +|t_2|\prod_{j=1}^m{{\rm Lip}(\Psi_j)}\to 0.
   \end{eqnarray*}
   The first term tends to 0 since the series $\sum_{k\geq 0}{(\overline{B}_{k+1}\vee|\underline{B}_{k+1}|)\prod_{j=1}^k{{\rm Lip}(\Psi_j)} }$ is convergent, 
   for details see \cite{vervaat1979stochastic}, 
   and the last two terms tend to 0 by the strong law of large 
   numbers for the sequence $\{ \log({\rm Lip}(\Psi_n)) \}_{n \geq 0}$.
   If we take $t_1=t_2=R_0$ we see that $\{ \Psi_{1:n}(R_0)  \}_{n \geq 0}$ is convergent and if we take $t_1=0$, $t_2=R_0$ and $n=m$ we see that the limit does not depend on $R_0$, 
   hence the stationary distribution is unique and it is the distribution of random variable
   \begin{equation}
      R = \lim_{n \to \infty}{\Psi_{1:n}(R_0)}.
   \end{equation} 
   For the rest of the proof we will assume that $R$ is given by the limit above. \bigskip
 
\textit{Step 2: Upper bound in \eqref{WeakR}.}
   We claim that
   \begin{equation}\label{MainClaim}
      \PP\left[\overline{R} > e^x, \: M\leq \log(\varepsilon)+ x \right] = \varepsilon^{\gamma/4} O(\PP[M >x])
   \end{equation}
   for $\varepsilon \in (0,1)$  sufficiently small and $\gamma>0$ given in the condition~\eqref{MomentsNew}. 
   To prove \eqref{MainClaim}, we will apply the technique from \cite{buraczewskiprecise, buraczewski2013}. For $k \in \ZZ$ define random set of integers by
   \begin{equation*}
      \Qcal(k):= \left\{ s \in \NN\: \left| \: \overline{B}_{s+1}\prod_{j=1}^s{A_j}  \in \left(e^{-k}e^x,e^{-k+1}e^x\right]\right\}\right..
   \end{equation*}
   Notice that if $M \leq \log(\varepsilon) +x$ then $\Qcal(k) =\emptyset$ for $k$ satisfying $e^{-k}> \varepsilon$. The following inclusion holds 
   \begin{equation}\label{inclusion}
      \left\{ \overline{R}> e^x,\: M \leq \log(\varepsilon)+ x \right\} \subseteq \left\{ \exists k : \: e^{-k} \leq \varepsilon, \: \# \Qcal(k) > \frac{e^k}{5k^2} \right\}.
   \end{equation}
   Indeed, assume that $\overline{R}> e^x$, $M \leq \log(\varepsilon)+ x$  and that for any $k$ such that $e^{-k} \leq \varepsilon$ we have $ \#\Qcal(k) \leq \frac{e^k}{5k^2}$. 
   Since $\Qcal(k)= \emptyset$ for $k$ satisfying $e^{-k} > \varepsilon$, we can write 
   \begin{eqnarray*}
      \overline{R} & = & \sum_{n\geq 0} {\overline{B}_{n+1} \prod_{j=1}^n{A_j}} = \sum_{k \in \ZZ}{ \sum_{ s \in \Qcal(k)}{\overline{B}_{s+1}\prod_{j=1}^s{A_j} } }\\
          & = & \sum_{k \geq - \log(  \varepsilon) }{ \sum_{ s \in \Qcal(k)}{ \overline{B}_{s+1}\prod_{j=1}^s{A_j}  }  } \leq \sum_{k > 0}{ \# \Qcal(k) e^{-k+1}e^x} \leq
                 \sum_{k>0}{ e^x\frac{e}{5k^2}} = \frac{\pi^2e}{30}e^x<e^x.
   \end{eqnarray*}
   This is a contradiction. Using the inclusion (\ref{inclusion}) one gets instantly that 
   \begin{equation} \label{inclusion2}
      \left\{ \overline{R}> e^x, M \leq \log(\varepsilon)+ x \right\}\subseteq 
         \left\{ M\leq\log(\varepsilon)+x,\:\exists k \geq - \log(\varepsilon),\:\# \Qcal(k)>\frac{e^k}{5k^2} \right\}.
   \end{equation}
   Let's focus our interest on the set $\mbox{RHS}(\ref{inclusion2})$. Define the~sequence $\tau(k)= \inf\Qcal(k)$ 
   (we use the convention that $\inf\emptyset =+ \infty$). 
   On the set $\mbox{RHS}(\ref{inclusion2})$ there exists $k \geq -\log(\varepsilon)$ such that $\tau(k) < \infty$ and from the fact that 
   $\tau(k) \in \Qcal(k)$ and $\# \Qcal(k) > \tfrac{e^k}{5k^2}$, we conclude that  
   \begin{equation*}
      \overline{B}_{\tau(k)+1}\prod_{j=1}^{\tau(k)}{A_j}, \quad \overline{B}_{\tau(k) +p+1}\prod_{j=1}^{\tau(k)+p}{A_j}\quad \in 
        \quad (e^{-k}e^x,e^{-k+1}e^x] \quad \mbox{for some } p > \frac{e^k}{5k^2}-1.
   \end{equation*}
   By taking $\varepsilon>0$ sufficiently small we can ensure that $\frac{e^k}{5k^2}-1> \frac{e^{k}}{10k^2}$ for $k \geq -\log(\varepsilon)$. By dividing the two quantities above 
   we obtain that
   \begin{equation}\label{bound1}
      \frac{A_{\tau(k)+1}}{\overline{B}_{\tau(k)+1}} \overline{B}_{\tau(k)+p+1} \prod_{j=\tau(k)+2}^{\tau(k)+p}{A_j}\quad \in \quad (e^{-1},e^{1}) \quad \mbox{for some } p > \frac{e^k}{10k^2}.
   \end{equation}
   The quotient $A_{\tau(k) +1}/\overline{B}_{\tau(k)+1}$ is bounded on the set $\mbox{RHS}(\ref{inclusion2})$, because
   \begin{equation}\label{bound2}
      \frac{A_{\tau(k)+1}}{\overline{B}_{\tau(k)+1}} = \frac{\prod_{j=1}^{\tau(k)+1}{A_j}}{\overline{B}_{\tau(k)+1}\prod_{j=1}^{\tau(k)}{A_j}} 
           \leq \frac{\overline{B}_{\tau(k)+2}\cdot \prod_{j=1}^{\tau(k)+1}{A_j}  }{e^{-k}e^x}
            \leq \frac{e^{M}}{e^{-k} e^x} \leq \frac{\varepsilon e^{x}}{e^{-k}e^x } \leq e^{k}.
   \end{equation}
   Combining bounds in \eqref{bound1} and \eqref{bound2} we can conclude: on the set $\mbox{RHS}(\ref{inclusion2})$ there exists an~integer $k \geq -\log(\varepsilon)$ 
   for which $\tau(k) < \infty$ and
   \begin{equation*}
      \overline{B}_{\tau(k) +p+1} \prod_{j=\tau(k)+2}^{\tau(k)+p}{A_j} > e^{-k-1} \quad \mbox{for some } p > \frac{e^k}{10k^2}.
   \end{equation*}
   Whence
   \begin{equation*}
           \log\left(\overline{B}_{\tau(k) +p+1}\right) + \sum_{l=\tau(k)+2}^{\tau(k)+p}{\frac{\mu}{2}+ \log( A_l)  } > \frac{\mu}{2}\left(\frac{e^k}{10k^2}-1\right) -k-1 
   \end{equation*}
   from which me may infer that 
   \begin{equation*}
      M_k^*=\sup_{j\geq\tau(k)+2} \left\{\log\left(\overline{B}_{ j}\right)+\sum_{l=\tau(k)+2}^{j-1}{\frac{\mu}{2}+\log(A_l)}\right\} > 
            \frac{\mu}{2}\left(\frac{e^k}{10k^2}-1\right) -k-1>e^{k/2} 
   \end{equation*}
   if $\varepsilon$ is small enough (recall that $k > - \log(\varepsilon)$). 
   So the following inclusion is also correct 
   \begin{equation}\label{InclusionCorrect}
      \left\{ \overline{R}> e^x, \: M\leq \log(\varepsilon)+ x  \right\} \subseteq 
          \bigcup_{ k\geq - \log(\varepsilon)} \left\{\tau(k)<\infty, \: M_k^*> e^{k/2}\right\}.
   \end{equation}
   Notice that by the strong Markov property, conditioned in $\tau(k)$, the distribution of
   $M_k^*$ is the same as the distribution of 
   \begin{equation*}
      M^*=\sup_{j\geq 2} \left\{\log\left(\overline{B}_{ j}\right)+\sum_{l=2}^{j-1}{\frac{\mu}{2}+\log(A_l)}\right\}. 
   \end{equation*}
   Theorem~\ref{ZPBZ} says that
   \begin{equation}\label{TailMB*}
       \PP\left[M^*>x\right] \sim \frac{2}{\mu} \int_x^{\infty}{ \PP[\log(A\vee B) > y] \: \ud y}\sim \frac{2}{\mu}\Fbar_I(x)\leq c_1x^{-\gamma},
   \end{equation} 
   for some $c_1>0$.
   In terms of probability~\eqref{InclusionCorrect} yields
   \begin{multline*}
      \PP\left[\overline{R}>e^x,\: M \leq \log(\varepsilon) +x\right] 
      \leq \sum_{ k \geq - \log( \varepsilon)  }{ \PP\left[\tau(k)<\infty, \: M_k^* > e^{k/2}\right]} \\
      = \sum_{k\geq-\log(\varepsilon) }{ \PP\left.\left[ M_k^* >e^{k/2}  \right| \: \tau(k) <\infty\right] \PP\left[\tau(k) <\infty\right]}  
      = \sum_{k \geq -\log( \varepsilon)}{ \PP\left[ M^*> e^{k/2}\right] \PP\left[\tau(k) <\infty\right]}
   \end{multline*}
   by the strong Markov property of the sequence $\{ (A_n,\overline{B}_n)\}_{n \geq 0}$.
   Using \eqref{TailMB*} we obtain for $\eta > 0$
   \begin{eqnarray*}
      \PP\left[\overline{R}>e^x,\: M\leq \log(\varepsilon)+x\right] & \leq & c_1\sum_{k\geq-\log(\varepsilon)}{\PP[\tau(k)< \infty]} e^{-k\gamma/2}
         \leq   c_1\hspace{-0.2 cm}\sum_{k\geq-\log(\varepsilon)}{\PP[M> x-k]e^{-k\gamma/2}} \\
         & \leq&  c_1 \hspace{-0.6 cm }\sum_{x-\eta\geq k\geq-\log(\varepsilon)}{\PP\left[M> x-k\right]e^{-k\gamma/2}} + c_1 \hspace{-0.1 cm}\sum_{k >x-\eta}{\PP[M> x-k]e^{-k\gamma/2}}  \\
         & =:  & c_1 I_1(x)+c_1I_2(x).
   \end{eqnarray*}
   Now we will investigate $I_1$ and $I_2$ separately. From the Theorem~\ref{ZPBZ} we can conclude that the distribution of the random variable $M$ belongs to the class 
   $\Scal \subseteq \Lcal$ and so we can use
   Potter bounds (Corollary \ref{Potter}) for $\PP\left[M>t\right]$  to find $\eta> 0$,  such that for $t,s>\eta$ we have
   \begin{equation*}
      \frac{\PP\left[M>t\right]}{\PP\left[M>s\right]} \leq 2 \exp\left\{\gamma\frac{|t-s|}{4}\right\}.
   \end{equation*}
   Then for $x>\eta-\log(\varepsilon)$ we have 
   \begin{equation*}
      \frac{I_1(x)}{\PP\left[M>x\right]} = \sum_{x-\eta\geq k\geq-\log(\varepsilon)}{\frac{\PP\left[M> x-k\right]}{\PP\left[M>x\right]}e^{-\gamma k/2}} \leq 
          2\sum_{x-\eta\geq k\geq-\log(\varepsilon)}{e^{-\gamma k/4}} 
        \leq C \varepsilon^{\gamma/4}
   \end{equation*}
   and for the second term
   \begin{equation*}
      I_2(x) \leq \sum_{k >x-\eta}{e^{-k\gamma/2}} \leq c_2 e^{-x\gamma/2} = o\left(\PP\left[M>x\right]\right)
   \end{equation*}
   for some $c_2>0$, since the distribution of $M$ is long-tailed. Thus claim \eqref{MainClaim} follows. Now we need notice that since $R \leq \overline{R}$ we have
   \begin{equation*}
      \left\{R>e^x\right\} \subseteq \left\{ \overline{R}>e^x,\: M\leq \log(\varepsilon)+ x\right\}\cup \left\{ M>\log(\varepsilon) +x\right\}
   \end{equation*}
   and thus using \eqref{MainClaim}, we get
   \begin{equation*} 
      \frac{\PP[R>e^x]}{\PP[M>x]} \leq \varepsilon^{\gamma/4} \frac{O(\PP[M>x])}{ \PP[M>x]} + \frac{\PP[M>x+\log(\varepsilon)]}{ \PP[M>x]}.
   \end{equation*}
   First let $x \to \infty$ and notice that from Theorem \ref{ZPBZ}
   \begin{equation}\label{TailMB}
       \PP\left[M>x\right] \sim -\frac{1}{ \EE[\log(A)]} \int_x^{\infty}{ \PP[\log(A\vee B) > y] \: \ud y}\sim \frac{1}{\mu}\Fbar_I(x).
   \end{equation}
   From this we can conclude that
   \begin{equation*}
      \limsup_{x \to \infty}  \frac{\PP[R>e^x]}{ \Fbar_I(x)} \leq C\varepsilon^{\gamma/4}  +\frac{1}{\mu}
   \end{equation*}
   for some finite constant $C>0$ independent of $\varepsilon >0$. 
   Since $\varepsilon>0$ is arbitrary small we get the upper bound. \bigskip

\textit{Step 3: Lower bound in \eqref{WeakR}.}
   Fix $0<\varepsilon$ and $ 0<\delta< \frac{\mu}{2}\wedge 1$ . For $K>0$ consider the sets $E_n$ and $F_n$ given by \eqref{EnDef} and \eqref{FnDef} respectively. 
   Choose $K>0$ large enough for \eqref{ConeProp} to be satisfied.
   Consider also the random variables 
   \begin{equation}\label{RnDef}
      R_n^* = \lim_{N \to \infty}{ \Psi_n\circ \ldots \circ \Psi_N(R_0)}.
   \end{equation}
   Note that $R_n^* \stackrel{d}{=}R$ and
   \begin{equation*}
      R= \Psi_{1:n+1}\left(R^*_{n+2}\right).
   \end{equation*}
   Finally put
   \begin{equation*}
      G_n = E_n\cap F_n \cap \left\{ A_{n+1} \vee \underline{B}_{n+1}> e^{n(\mu +\delta)+ L+K +x}, \: \underline{B}_{n+1} \geq -e^{n(\mu-\delta)-K+ x }  \right\} \cap 
        \left\{ R^*_{n+2} > \delta  \right\}
   \end{equation*}
   where $L >0$ is a constant independent of $x$ and $n$.
   We see that the sets $\{ G_n\}_{n \geq 0}$ are disjoint if we take $L=L(K, \delta, \mu)$ sufficiently large. Moreover on the set $G_n$ we have
   \begin{eqnarray*}
      R & = & \Psi_{1:n+1}(R_{n+2}^*)\geq \underline{\Psi}_{1:n+1} (R^*_{n+2})  = 
                               \sum_{k=0}^{n-1}\underline{B}_{k+1}\prod_{j=1}^k{A_j} +\left(\underline{B}_{n+1} + R^*_{n+2} A_{n+1}\right) \prod_{j=1}^n{A_j} \\
        & \geq & -\sum_{k=0}^{n-1} \left|\underline{B}_{k+1}\right|\prod_{j=1}^k{A_j} +\left( \underline{B}_{n+1}+e^{n(\mu-\delta)-K+x} +A_{n+1}R^*_{n+2}\right) \prod_{k=1}^{n}{A_k}-
               e^{n(\mu-\delta)-K+x}\prod_{j=1}^n{A_j}  \\ 
        & \geq &-\frac{e^{2K}}{1-e^{-\mu+2\delta}}  + \delta\left(A_{n+1}\vee \left(\underline{B}_{n+1} +e^{n(\mu-\delta) -K +x}\right)\right) e^{-n(\mu+\delta)-K}-e^{x} \\
            & \geq &-\frac{e^{2K}}{1-e^{-\mu+2\delta}}  + \delta(A_{n+1}\vee \underline{B}_{n+1}) e^{-n(\mu+\delta)-K}-e^{x} 
                               \geq-\frac{e^{2K}}{1-e^{-\mu+2\delta}} + \delta e^{x+L}-e^{x}  > e^x
   \end{eqnarray*}
   and the last inequality is valid for all $x>0$ and all $n \in\NN$ if $L=L(K,\delta,\mu)$ is sufficiently large. We see that 
   $G_n \subseteq \left\{ R> e^x \right\}$ and this allows us to write
   \begin{multline*}
      \PP[R>e^x]  \geq \sum_{n \geq 0}{\PP[G_n]}\\ 
                  \geq (1-\varepsilon) \sum_{n \geq 0}{\PP\left[ A_{n+1}\vee \underline{B}_{n+1}> e^{n(\mu +\delta)+L+x+K}, \: \underline{B}_{n+1} \geq -e^{x+n(\mu-\delta)-K} \right] 
                                 \PP\left[R^*_{n+2} > \delta\right]  } \\
                  \geq(1-\varepsilon) \PP[R>\delta] \sum_{n \geq 0} \left\{\PP\left[ A\vee \underline{B} > e^{n(\mu +\delta)+L+K +x}   \right] \right. \\ -
                                                                          \left. \PP\left[A>e^{n(\mu +\delta)+L+K+x}, \: \underline{B}<-e^{x+n(\mu-\delta)+K}  \right] \right\}    \\
                  \sim \frac{(1-\varepsilon)\PP[R>\delta]}{\mu+\delta} \int_x^{\infty}{ \PP[\log(A\vee B) > y] \: \ud y}.
   \end{multline*}
   This yields 
   \begin{equation*}
      \liminf_{x \to \infty}{ \frac{\PP\left[R>e^x\right]}{\int_x^{\infty}{ \PP[\log(A\vee B) > y] \: \ud y} }} \geq \frac{(1-\varepsilon)\PP[R>\delta]}{\mu +\delta}.
   \end{equation*}
   If we allow $\varepsilon, \delta \to 0$ we see that we have proven the lower estimate for the desired limit. \bigskip

\textit{Step 4: The case $\PP[A>x] = o(\PP[B>x])$.}
   Firstly, notice that we need only to prove the lower bound and that in this case Theorem~\ref{ZPBZ} yields 
   \begin{equation}\label{TailMBB}
       \PP\left[M >x\right] \sim -\frac{1}{\EE[\log(A)]}\int_x^{\infty}{ \PP[\log^+(B) > y] \: \ud y}.
   \end{equation}
   For $0<\varepsilon$, $0< \delta < \mu/2$ and $K>0$ consider the sets $E_n$ and $F_n$ given by \eqref{EnDef} and \eqref{FnDef} respectively with $K>0$ large enough for \eqref{ConeProp} to 
   be satisfied. 
   Finally, put
   \begin{equation*}
      J_n = E_n\cap F_n \cap \left\{ \underline{B}_{n+1}>e^{x+n(\mu+\delta) +K+L},\: A_{n+1}\leq e^{n(\mu -\delta)-K +x}  \right\} \cap \left\{ |R_{n+2}^*| \leq \delta^{-1} \right\}.
   \end{equation*}
   For some large $L>0$ independent of $x$.  We see that the sets $\{ J_n\}_{n \geq 0}$ are disjoint. Moreover on the set $J_n$ we have
   \begin{eqnarray*}
      R & =    & \Psi_{1:n+1}(R_{n+2}^*) \geq \underline{\Psi}_{1:n+1} (R^*_{n+2}) = 
                                 \sum_{k=0}^{n-1}\underline{B}_{k+1}\prod_{j=1}^k{A_j} +\underline{B}_{n+1}\prod_{j=1}^n{A_j}+ R^*_{n+2} A_{n+1} \prod_{j=1}^n{A_j} \\
        & \geq & -\sum_{k=0}^{n-1} \left|\underline{B}_{k+1}\right|\prod_{j=1}^k{A_j} +\underline{B}_{n+1}\prod_{j=1}^n{A_j}  - A_{n+1}|R^*_{n+2}| \prod_{k=1}^{n}{A_k} \\ 
        & \geq & -\frac{e^{2K}}{1-e^{-\mu+2\delta}} + e^{x+L}  - \delta^{-1}e^x > e^x
   \end{eqnarray*}
   and the last inequality is valid for all $x>0$ if $L=L(K, \delta, \mu)$ is sufficiently large. Therefore $J_n \subseteq \left\{ R> e^x \right\}$ and this allows us to write
   \begin{multline*}
      \PP[R>e^x] \geq \sum_{n \geq 0}{\PP[J_n]}\\  \geq 
     (1-\varepsilon) \sum_{n \geq 0}{\PP\left[ \underline{B}_{n+1}>e^{x+n(\mu+\delta) +K+L},\: A_{n+1}\leq e^{n(\mu -\delta)-K +x}   \right] \PP\left[|R_{n+2}^*| \leq \delta^{-1} \right]}\\\geq
     (1-\varepsilon) \PP\left[|R|\leq \delta^{-1}\right] \sum_{n \geq 0}{\left\{\PP\left[ \underline{B}>e^{x+n(\mu+\delta)+K+L}\right]-\PP\left[A> e^{n(\mu -\delta)-K+x}\right]\right\}} \\
  \sim   \frac{(1-\varepsilon)\PP\left[|R|\leq \delta^{-1} \right]}{\mu+\delta} \int_x^{\infty}{ \PP[\log^+(B) > y] \: \ud y}. 
   \end{multline*}
   This yields 
   \begin{equation*}
      \liminf_{x \to \infty}{ \frac{\PP\left[R>e^x\right]}{ \int_x^{\infty}{ \PP[\log^+(B) > y] \: \ud y} }} \geq \frac{(1-\varepsilon) \PP\left[|R|\leq \delta^{-1} \right]}{\mu +\delta}.
   \end{equation*}
   If we allow $\varepsilon, \delta \to 0$ we get the lower estimate.\bigskip

\textit{Step 5: The case $\PP[B>x] = o(\PP[A>x])$.} 
   Notice that we only need to prove the upper estimate and that in this case
   \begin{equation}\label{TailMBA}
      \PP\left[ M >x \right]  \sim -\frac{1}{ \EE[\log(A)]} \int_x^{\infty}{ \PP[\log(A) > y] \: \ud y}.
   \end{equation}
   Fix $\varepsilon \in (0.1)$ and notice that since \eqref{MainClaim} holds we only need to focus on the set $\mbox{LHS}(\ref{TheOtherSet})$:
   \begin{equation}\label{TheOtherSet}
      \left\{R>e^x, M > \log(\varepsilon)+ x  \right\} \subseteq \left\{ M \in ( \log(\varepsilon)+x, x]  \right\}\cup 
            \left\{R>e^x, M > x  \right\}
   \end{equation}
   and since the distribution of $M$ is long-tailed and \eqref{TailMBA} is valid we have 
   \begin{equation}\label{MBLongTail}
      \PP\left[M \in ( \log(\varepsilon)+ x, x] \right] = o\left(\int_x^{\infty}{ \PP[\log(A) > y] \: \ud y} \right).
   \end{equation}
   For the other set we have 
   \begin{equation*}
      \left\{R>e^x, M > x  \right\} = \left\{ M>  x \right\} \setminus \left\{R\leq e^x, M > x  \right\}
   \end{equation*}
   so by \eqref{TailMBA} now we only need to prove that
   \begin{equation*}
      \liminf_{x \to \infty} \frac{\PP \left[R\leq e^x, M > x \right]  }{ \int_x^{\infty}{ \PP[\log(A > y] \: \ud y} } \geq -\frac{ \PP[R \leq 0]}{\EE[\log(A)]}.
   \end{equation*}
   We achieve that using the same technique, but this time we consider the sets
   \begin{equation*}
      H_n = E_n\cap F'_n \cap \left\{ \overline{B}_{n+1}\leq \frac{1}{2}e^{x +n(\mu-\delta) -K},\: A_{n+1}> e^{n(\mu +\delta)+K +x}  \right\} \cap \{ R^*_{n+2} \leq 0  \},
   \end{equation*}
   where
   \begin{equation*}
      F_n'=F_n'(\delta,K)=\left\{\left| \overline{B}_{j} \right| < e^{\delta j +K}, j \leq n  \right\}.
   \end{equation*}
   Note that \eqref{ConeProp} also holds true if we replace $F_n$ by $F_n'$.
   We see that the sets $\{ H_n\}_{n \geq 0}$ are disjoint if $x$ is sufficiently large. Moreover on the set $H_n$ we have
   \begin{eqnarray*}
      R & = & \Psi_{1:n+1}(R^*_{n+2})\leq \overline{\Psi}_{1:n+1}(R^*_{n+2})  =
          \sum_{j=0}^{n-1}\overline{B}_{j+1}\prod_{k=1}^j{A_k} + \overline{B}_{n+1}\prod_{k=1}^n{A_k}  
          +(R^*_{n+2})^+ \prod_{k=1}^{n+1}{A_k} \\  & \leq&  \frac{e^{2K}}{1-e^{-\mu+2\delta}} +\frac{1}{2}e^{x} + 0 \leq e^x
   \end{eqnarray*}
   and the last inequality is valid for all $x >x_0 =x_0(K,\delta,\mu)$.
   Therefore $H_n \subseteq \left\{ R \leq e^x \right\}$. Moreover on the set $H_n$
   \begin{equation*}
      M \geq \sum_{k=1}^{n+1}{\log(A_k)} > -n (\mu +\delta) -K +n(\mu +\delta)+K+x =x
   \end{equation*}
   and this proves that $H_n \subseteq \left\{ R\leq e^x, \: M>x  \right\}$, which allows us to write
   \begin{multline*}
     \PP\left[R\leq e^x, M>x\right] \geq \sum_{n \geq 0}{\PP[H_n]} \\ \geq 
     (1-\varepsilon) \sum_{n \geq 0}{\PP\left[ \overline{B}_{n+1}\leq \frac{1}{2}e^{x+n(\mu-\delta)-K},\: A_{n+1}> e^{n(\mu +\delta)+K+x}   \right] \PP[\overline{R}_{n+2} \leq 0 ]  } \\ 
     \geq (1-\varepsilon) \PP[R \leq 0] \sum_{n \geq 0}{\left\{\PP\left[ A>e^{x+n(\mu+\delta)+K}\right]- \PP\left[ \overline{B}> \frac{1}{2}e^{n(\mu -\delta)+x-K}   \right]\right\}} \\  
     \sim \frac{(1-\varepsilon)\PP[R \leq 0]}{\mu+\delta}\int_x^{\infty}{ \PP[\log(A) > y] \: \ud y}.
   \end{multline*}
   This yields 
   \begin{equation}\label{JustProved}
      \liminf_{x\to\infty}{\frac{\PP\left[R\leq e^x, M>x \right]}{\int_x^{\infty}{ \PP[\log(A) > y] \: \ud y} }} \geq \frac{(1-\varepsilon)\PP[R \leq 0]}{\mu +\delta}.
   \end{equation}
   So if we put everything together, we notice that since $R \leq \overline{R}$ we have
   \begin{multline*}
      \{R>e^x\}\subseteq\left\{\overline{R}>e^x,\: M\leq\log(\varepsilon)+x\right\}\cup\left\{M\in(\log(\varepsilon)+x,x]\right\}\cup \\ 
           \cup \left( \left\{ M > x  \right\}\setminus \left\{R\leq e^x,M > x  \right\} \right)
   \end{multline*}
   and thus
     \begin{equation*}
      \PP[R>e^x] \leq \PP\left[\overline{R}>e^x,\: M\leq\log(\varepsilon)+x\right]+\PP\left[M\in(\log(\varepsilon)+x,x]\right]+ 
            \PP\left[ M > x  \right] - \PP\left[R\leq e^x,M > x  \right]
   \end{equation*} 
   and so using \eqref{MainClaim}, \eqref{MBLongTail}, \eqref{TailMBA} and \eqref{JustProved} we get
   \begin{equation*}
     \limsup_{x \to \infty}{ \frac{\PP\left[R > e^x \right]}{\int_x^{\infty}{ \PP[\log(A) > y] \: \ud y} }} \leq C\varepsilon^{\gamma/4} + 0 + \frac{1}{\mu} - 
        \frac{(1-\varepsilon)\PP[R\leq 0]}{\mu + \delta}.
   \end{equation*}
   If we allow $\varepsilon, \delta \to 0$ we see that we achieved the desired upper bound and hence the proof is complete. 
\end{proof}

   Now we can turn our attention to the finite time horizon. Notice that Theorem~\ref{PropFin} follows by induction form the following lemma.

   \begin{lem}
      Assume \eqref{ConditionsOptimal}, \eqref{Standing}, \eqref{Tail} and that $F$ defined by \eqref{FDef} is subexponential. 
      Assume additionally that
      \begin{equation*}
         w_1=\liminf_{x\to \infty}\frac{\PP[R_0>x]}{\PP[A \vee B > x]}\leq \limsup_{x\to \infty}\frac{\PP[R_0>x]}{\PP[A \vee B > x]}=w_2
      \end{equation*}
      for some finite constants $w_1,w_2 \geq 0$. Then, for $R_1 = \Psi_1(R_0)$
      \begin{equation}\label{WeakFinLemm}
         w_1+\PP[R_0>0] \leq\liminf_{x\to\infty}\frac{\PP[R_{1}>x]}{\PP[A \vee B > x]}\leq\limsup_{x \to\infty}{\frac{\PP[R_{1}>x]}{\PP[A \vee B >x ]}} \leq 1 + w_2.
      \end{equation}
      Furthermore if
      \begin{itemize}
         \item $\PP[A>x] = o (\PP[B>x])$ then
               \begin{equation}
         w_1+1\leq\liminf_{x\to\infty}\frac{\PP[R_{1}>x]}{\PP[ B > x]}\leq\limsup_{x \to\infty}{\frac{\PP[R_{1}>x]}{\PP[ B >x ]}} \leq 1 + w_2.
               \end{equation}

         \item $\PP[B>x] = o (\PP[A>x])$ then
               \begin{equation}
      w_1+\PP[R_0>0]\leq\liminf_{x\to\infty}\frac{\PP[R_{1}>x]}{\PP[A  > x]}\leq\limsup_{x \to\infty}{\frac{\PP[R_{1}>x]}{\PP[A  >x ]}} \leq \PP[R_0>0]+ w_2.
               \end{equation}

      \end{itemize}       
   \end{lem}

\begin{proof}
   The proof mimics the one of the main result. Fix $x \in \RR$.\bigskip

\textit{Step 1: Upper bound in \eqref{WeakFinLemm}.} Notice that
   \begin{equation*}
      R_1  =  \Psi_{1}(R_0)  \leq  \overline{\Psi}_{1}(R_0) = \overline{B}_{1} + R^+_0A_1 \leq A_{1}\vee \overline{B}_{1} + R^+_0 (A_1\vee \overline{B}_1) 
                                    \leq  (1+R^+_0)(A_1\vee \overline{B}_1)
   \end{equation*}
   and so
   \begin{equation*}
      \PP[R_1 >x] \leq \PP\left[ (1+R^+_0)(A_1\vee \overline{B}_1) > x \right]\leq (1+w_2+o(1)) \PP[A\vee B > x]
   \end{equation*}
   as $x \to \infty$, since $F$ is subexponential and Proposition~\ref{PropFoss} holds. \bigskip

\textit{Step 2: Lower bound in \eqref{WeakFinLemm}.}
   Fix $0<\varepsilon$ and $ 0<\delta< \frac{\mu}{2}\wedge 1$ . For $K,L>0$ consider the sets 
   \begin{equation*}
      G_0 = \left\{ A_{1} \vee \underline{B}_{1}> e^{L+K +x}, \: \underline{B}_{1}\geq -e^{-K+x}\right\}\cap\left\{ R_0 > \delta  \right\}
   \end{equation*}
   and
   \begin{equation*}
      G_1 = \left\{ e^{-K}\leq A_1 \leq e^K, \: \left|\underline{B}_1\right|\leq e^{K} \right\}\cap \left\{R_0 > e^{L+K +x}  \right\}.
   \end{equation*}
   Take $K>0$ sufficiently large such that
   \begin{equation}\label{ConeOne}
      \PP\left[ e^{-K}\leq A_1 \leq e^K, \: \left|\underline{B}_1\right|\leq e^{K}  \right] \geq 1-\varepsilon.
   \end{equation}
   We see that the sets $G_0$ and $G_1$ are disjoint if we take $L=L(K, \delta, \mu)$ large enough. Moreover on the set $G_0$ we have
   \begin{eqnarray*}
     R_1 & = & \Psi_{1}(R_0)\geq  \underline{\Psi}_{1}(R_0)  
          =                \underline{B}_{1} + A_{1} R_0 \\
     & = & \left(\underline{B}_{1}+e^{-K+x}+A_{1}R_0\right)-e^{-K+x}  
          \geq             \delta(A_{1}\vee \underline{B}_{1}) -e^x \geq \delta e^{x+L} -e^x > e^x
   \end{eqnarray*}
   and the last inequality is valid for all $x>0$  if $L=L(K,\delta,\mu)$ is sufficiently large. 
   On the set $G_1$ we have   
   \begin{eqnarray*}
     R_1 & = & \Psi_{1}(R_0) \geq \underline{\Psi}_{1}(R_0) = \underline{B}_{1} + R_0A_1 
          \geq  -|\underline{B}_{1}|+R_0A_1 \\ 
         & \geq            & -e^K  + R_0 e^{-K} \geq -e^K + e^{x+L}   > e^x
   \end{eqnarray*}
   and again, the last inequality holds if we take $L = L(K,\delta, \mu)$ sufficiently large.
   Therefore $G_0\cup G_1 \subseteq \left\{ R_1> e^x \right\}$ and  since 
   \begin{eqnarray*}
      \PP[G_0] & \geq & \PP\left[ A_{1}\vee \underline{B}_{1}> e^{ L+x+K}, \: \underline{B}_{1}\geq -e^{-K+x} \right] 
                             \PP[ R_0 > \delta]   \\
               & \geq & \PP[ R_0> \delta] \left(\PP\left[ A_{1}\vee \underline{B}_{1}> e^{ L+x+K} \right]\right. 
                      -\left.\PP\left[ A_{1}> e^{ L+x+K}, \: \underline{B}_{1}> -e^{ -K+x} \right] \right)\\
               & \sim & \PP[R_0>\delta]\PP\left[ A\vee B > e^{x}   \right]  
   \end{eqnarray*}
   and
   \begin{equation*}
      \PP[G_1] \geq (1-\varepsilon) \PP\left[R_0 > e^{L+K +x}   \right] \geq  ((1-\varepsilon)w_1+o(1))\PP[A\vee B>e^x].
   \end{equation*}
   We can write 
   \begin{equation*}
      \PP[R_1>e^x] \geq \PP[G_0] +\PP[G_1] \geq \left((1-\varepsilon)  w_1 + \PP[R_0>\delta] + o(1) \right) \PP[A\vee B >e^x].
   \end{equation*}
   This yields 
   \begin{equation*}
      \liminf_{x \to \infty}{ \frac{\PP\left[R_n>e^x\right]}{ \PP[ A\vee B > e^x] }} \geq \left((1-\varepsilon)w_1+\PP[R_{0} > \delta ] \right) .
   \end{equation*}
   If we allow $\varepsilon, \delta \to 0$ we see that we have proven the lower estimate for the desired limit. \bigskip

\textit{Step 3: The case $\PP[A>x] = o(\PP[B>x])$.}
   For $0<\varepsilon$, $0< \delta < \mu/2$ and $K, L>0$ consider the sets 
   \begin{equation*}
      J_0 = \left\{\underline{B}_{1}>e^{x +K+L},\: A_{1}\leq e^{ -K +x}  \right\} \cap \left\{ |R_0| \leq \delta^{-1}  \right\}.
   \end{equation*}
   and 
   \begin{equation*}
      J_1 = \left\{ e^{-K}\leq A_1 \leq e^K, \: \left|\underline{B}_1\right|\leq e^{K} \right\} \cap \left\{ R_0>e^{x+K+L} \right\}
   \end{equation*}    
   with $K$ such that~\eqref{ConeOne} is satisfied. We see that the sets $J_0$ and $J_1$ are disjoint. Moreover on the set $J_0$ 
   \begin{equation*}
      R_{1}  =   \Psi_{1}(R_0) \geq \underline{\Psi}_{1}(R_0) 
             =                 \underline{B}_{1} + R_0A_{1}        
             \geq              \underline{B}_{1} - |R_0|A_{1}   
             \geq              e^{x+L} - \frac{1}{\delta}e^x > e^x
   \end{equation*}
   and on the set $J_1$
   \begin{equation*}
      R_{n}  =   \Psi_{1}(R_0) \geq \underline{\Psi}_{1}(R_0) = \underline{B}_{1} + R_0A_1
             \geq   -|\underline{B}_{1}| + R_0A_1    
             \geq  -e^K + e^{x+L}  > e^x
   \end{equation*}
   and the last inequality is valid for all $x>0$ if $L=L(K, \delta, \mu)$ is sufficiently large. Therefore $J_0\cup J_1 \subseteq \left\{ R> e^x \right\}$.  
   \begin{eqnarray*}
     \PP[J_0]  & \geq & \PP\left[ \underline{B}_{1}>e^{x +K+L},\: A_{1}\leq e^{-3K +x}   \right] \PP\left[|R_0| \leq \delta^{-1}\right]  \\ 
               & =    &  \PP\left[|R_0| \leq \delta^{-1}\right]\left(\PP\left[\underline{B}>e^{x+K+L}\right]- \PP\left[ A> e^{-K +x}\right] \right)   \\ 
               & \geq & (\PP\left[|R_0| \leq \delta^{-1} \right] + o(1) )\PP[B > e^x] 
   \end{eqnarray*}
   and 
   \begin{equation*}
      \PP[J_1] \geq (1-\varepsilon) \PP[R_0 > e^{x+K+L}] \geq ((1-\varepsilon)w_1+o(1))\PP[B>e^x].
   \end{equation*} 
   This allows us to write 
   \begin{equation*}
      \PP[R_1>e^x] \geq \PP[J_0] +\PP[J_1]\geq \left((1-\varepsilon) w_1+\PP\left[|R_0| \leq \delta^{-1}\right]+ o(1)\right) \PP[B>e^x]
   \end{equation*}
   This yields 
   \begin{equation*}
      \liminf_{x \to \infty}{ \frac{\PP\left[R>e^x\right]}{\PP[\log^+(B) > x]} } \geq (1-\varepsilon) \left( w_1 + \PP\left[|R_0| \leq \delta^{-1}\right] \right).
   \end{equation*}
   If we allow $\varepsilon, \delta \to 0$ we get the lower estimate. \bigskip

\textit{Step 4: The case $\PP[B>x] = o(\PP[A>x])$.}
   Notice that we only need to prove the upper estimate. Let
   \begin{equation*}
      M_1 = \left\{ \log(\overline{B}_{1})\right\} \vee \left\{ \log^+(R_0) + \log(A_1) \right\}.
   \end{equation*}
   Next, notice that
   \begin{equation*}
      \left\{ \overline{\Psi}_{1}(R_0)>e^x, M_1 \leq x-\log(2) \right\} = \emptyset
   \end{equation*}
   and so 
   \begin{equation}\label{PartitionWN}
      \begin{array}{rcl}
      \left\{ \overline{\Psi}_{1}(R_0) > e^x\right\} & = &\left\{\overline{\Psi}_{1}(R_0)>e^x,\: M_1 > x-\log(2)\right\} \\
                                                       & = & \left\{ M_1>  x-\log(2) \right\} \setminus \left\{\overline{\Psi}_{1}(R_0)\leq e^x, M_1 > x-\log(2)\right\}.
      \end{array}
   \end{equation}
   Since $M_1\leq \log^+(R_0) + \log(A_1\vee \overline{B}_1)$ we can write
   \begin{equation}\label{MBTailBound}
      \begin{array}{rcl}
      \PP\left[M_1 > x-\log(2)  \right] & \leq  &\PP\left[\log^+(R_0)+\log(A_1\vee \overline{B}_1) >x-\log(2) \right] \\
                                          & \geq  & (w_1+1+o(1)) \PP[\log(A \vee B) > x]
      \end{array}
   \end{equation}
   so we only need to prove that
   \begin{equation*}
      \liminf_{x \to \infty} \frac{\PP \left[\overline{\Psi}_{1}(R_0)\leq e^x, M_1 > x-\log(2) \right]  }{ \PP[\log(A) > x] } \geq \PP[R_0\leq 0].
   \end{equation*}
   We achieve that using the same technique, but this time we consider the set
   \begin{equation*}
      H_0 = \left\{ \overline{B}_{1}\leq e^{x},\: A_{1}> e^{x}  \right\} \cap \{ R_0 \leq 0  \}.
   \end{equation*}
   on which we have
   \begin{equation*}
      \Psi_{1}(R_0)  \leq  \overline{\Psi}_{1}(R_0) = \overline{B}_{1} + R_0^+A_1  \leq  e^{x} + 0 \leq e^x
   \end{equation*}
   and
   \begin{equation*}
      M_1 \geq \log(A_1) > x.
   \end{equation*}
   We see that $H_1 \subseteq \{ \Psi_{1}(R_0)\leq e^x, \: M>x -\log(2) \}$ and this allows us to write
   \begin{multline*}
     \PP\left[\Psi_{1}(R_0)\leq e^x,\: M_1>x - \log(2)\right] \geq \PP[H_0] \geq 
     \PP\left[\overline{B}_{1}\leq e^{x},\: A_{1}> e^{x}\right]\PP[R_0 \leq 0]  \\ 
    \geq \left\{\PP\left[A_{1}> e^{x} \right] -  \PP\left[ \overline{B}_{1}> e^{x}  \right]  \right\} \PP[R_0 \leq 0]  \geq
     \PP[A> e^x] \left( \PP[ R_0 \leq 0 ] +o(1)\right).
   \end{multline*}
   This yields 
   \begin{equation*}
      \liminf_{x \to \infty}{ \frac{\PP\left[R_1\leq e^x, M_1 > x - \log(2)  \right]}{\PP[\log(A) > x] }} \geq \PP[R_{0} \leq 0 ] .
   \end{equation*}
   Putting everything together, that is \eqref{PartitionWN}, \eqref{MBTailBound} and the last inequality we get
   \begin{equation*}
      \limsup_{x \to \infty}{ \frac{\PP\left[R_1> e^x\right]}{\PP[\log(A) > x] }} \leq 1+w_2 -\PP[R_0 \leq 0 ]  
   \end{equation*}
   which is the desired upper bound and hence the proof is complete in this case. 
\end{proof}

\section*{Acknowledgments}
Part of the main result of this paper was included in author’s Master's thesis, written under the~supervision of Dariusz Buraczewski at the University of Wroclaw. 
The author would like to thank him for hours of stimulating conversations and several helpful suggestions during the~preparation of this paper. 
The author would like to thank also Zbigniew Palmowski, Tomasz Rolski and Bert Zwart for pointing out references and fruitful conversations as well as both referees for 
constructive suggestions for improving the presentation of this paper. 


\bibliographystyle{amsplain}

\providecommand{\bysame}{\leavevmode\hbox to3em{\hrulefill}\thinspace}
\providecommand{\MR}{\relax\ifhmode\unskip\space\fi MR }
\providecommand{\MRhref}[2]{%
  \href{http://www.ams.org/mathscinet-getitem?mr=#1}{#2}
}
\providecommand{\href}[2]{#2}

\end{document}